\documentclass[11pt]{amsart}

\usepackage{amsmath}
\usepackage{amssymb}
\usepackage{graphics,graphicx}
\usepackage{enumerate}
\usepackage{color,colordvi}

\newtheorem{theorem}{Theorem}[section]
\newtheorem{corollary}[theorem]{Corollary}

\newtheorem{proposition}[theorem]{Proposition}

\theoremstyle{definition}

\newtheorem{example}[theorem]{Example}
\newtheorem{remark}[theorem]{Remark}
\newtheorem{remarks}[theorem]{Remarks}

\newcommand{\D}{\mathrm{d}}
\newcommand{\e}{\mathrm{e}}

\newcommand{\N}{\mathbb{N}}
\newcommand{\R}{\mathbb{R}}
\newcommand{\Z}{\mathbb{Z}}
\newcommand{\BB}{\mathcal{B}}

\newcommand{\OO}{\mathcal{O}}
\newcommand{\RR}{\mathcal{R}}
\newcommand{\pd}{\partial}
\newcommand{\eps}{\varepsilon}

\numberwithin{equation}{section}

\title[Local perturbations of leaky curves]{Discrete spectrum from local perturbations of leaky curves}

\dedicatory{To Israel Sigal on the occasion of his 80th birthday}

\author[P.~Exner]{Pavel Exner}

\address[P.~Exner]{Doppler Institute for Mathematical Physics and Applied Mathematics\\
Czech Technical University\\
B\v rehov\'a 7\\ 11519 Prague\\ Czechia\\ and Department of
Theoretical Physics\\ NPI\\ Academy of Sciences\\ 25068 \v{R}e\v{z}
near Prague, Czechia}
\email{{\tt exner@ujf.cas.cz}}



\begin{document}


\begin{abstract}
We discuss spectrum of a class of singular Schr\"odinger operator models known as leaky curves and show that if the interaction support has a periodic shape, its local perturbations can give rise to a discrete spectrum below the continuum threshold even if they are of `zero mean'.
\end{abstract}

\maketitle

\section{Introduction}
\label{s:intr}

Being aware that main interests of the man we are celebrating are focused on many-body systems, I choose the topic of this paper with the hope that he would agree with the phrase I heard many years ago from the great physicist Mahito Kohmoto: `Surprisingly, there are still interesting effects in one-body quantum mechanics'. The system I am going to discuss is simple indeed; its Hamiltonian is a singular Schr\"odinger operator in $L^2(\R^2)$ formally given by
 \begin{equation} \label{formal}
H_{\alpha,\Gamma} = -\Delta-\alpha\delta(x-\Gamma), \quad \alpha>0,
 \end{equation}
where $\Gamma$ is a sufficiently regular curve. It is well known that the geometry of $\Gamma$ induces an effective interaction. For instance, if the curve is straight asymptotically but not globally, operator \eqref{formal} typically has at least one eigenvalue below the continuum starting at $-\frac14\alpha^2$ \cite{EI01}.

In the present paper we deal with local perturbations of \emph{periodic} curves. We prove sufficient conditions under which they give rise to eigenvalues at the bottom of the spectrum. Not surprisingly, it happens if the perturbation represents a local contraction, but we are going to show that this is possible even in the `critical' case when the perturbation is of zero mean.

\section{Preliminaries}
\label{s:prelim}

Let us now describe the model in more precise terms. We suppose that the curve $\Gamma:\R\to\R^2$ is Lipschitz and piecewise $C^1$; with an abuse of notation we use the same symbol for the map and its image. The formal expression \eqref{formal} can be then naturally given meaning using the quadratic form
 \begin{equation} \label{Hform}
q_{\alpha,\Gamma}: q_{\alpha,\Gamma}[\psi] = \int_{\R^2} |\nabla\psi(x)|^2\,\D x - \alpha \int_\R |\psi(\Gamma(s))|^2\,\D s,
 \end{equation}
where $s$ is the arc length of $\Gamma$. The domain of \eqref{Hform} is $H^1(\R^2)$ and the form is closed and bounded from below, cf.~\cite[Sec.~2]{BEKS94} or \cite[Sec.~2.1]{Ex08}, being thus uniquely associated with a self-adjoint operator for which we use again the symbol $H_{\alpha,\Gamma}$.

The unperturbed system will correspond to a periodic curve. For simplicity we restrict ourselves to the case when $\Gamma$ is a graph of a function; without loss of generality we may fix the Cartesian coordinates and assume that
 \begin{enumerate}[(i)]
 \setcounter{enumi}{0}
 \setlength{\itemsep}{1.5pt}
 \item $\Gamma_0(s)=(x_1,x_2=\gamma(x_1))$ with a Lipschitz function $\gamma:\R\to\R$, piecewise $C^1$ and such that $\gamma(x_1+a)=\gamma(x_1)$ for some $a>0$, where $x_1=x_1(s)$ is the inverse function to $s:\:s(x_1)= \int^{x_1}_0 \sqrt{1+\gamma'(\xi)^2}\,\D\xi$. We exclude the trivial case when $\gamma$ is a constant function. \label{ai}
  \end{enumerate}

\begin{proposition} \label{prop:ess}
Under the assumption \eqref{ai}, the spectrum of operator $H_{\alpha,\Gamma_0}$ is purely essential,
 $$ 
\sigma\big(H_{\alpha,\Gamma_0}\big) = \sigma_\mathrm{ess}\big(H_{\alpha,\Gamma_0}\big) \supset \R_+
 $$ 
with $\inf \sigma\big(H_{\alpha,\Gamma_0}\big)<0$; the negative part is absolutely continuous.
\end{proposition}
\begin{proof}
In view of the singular potential periodicity one can use Floquet decomposition: with an abuse of notation we write
 \begin{equation} \label{floquet}
H_{\alpha,\Gamma_0} = \frac{a}{2\pi} \int_{|\pi\theta|\le a}^\oplus H_{\alpha,\Gamma_0}(\theta)\,\D\theta,
 \end{equation}
where the right-hand side has to be understood as $UH_{\alpha,\Gamma_0}U^{-1}$ with the unitary $U:\,L^2(\R^2)\to L^2(S_a)\times [-\frac{\pi}{a},\frac{\pi}{a}]$ defined on $C_0^\infty(\R^2)$ by
 $$
(U\psi)(x,\theta) = \sum_{m\in\Z} \e^{-im\theta} \psi\Big(x_1+ma, x_2\Big)
 $$
and extended by continuity; the elementary cell here is the strip
 \begin{equation} \label{strip}
S_a = [-\textstyle{\frac12}a,\textstyle{\frac12}a] \times \R.
 \end{equation}
The fiber operators on $L^2(S_a)$ in \eqref{floquet} act as
 \begin{equation} \label{fiber}
H_{\alpha,\Gamma_0}(\theta) = -\Delta -\alpha\delta(x-\Gamma(s)),
 \end{equation}
their domains consisting of functions from $H^2(S_a)$ satisfying the conditions
 \begin{equation} \label{quasiper}
\pd_1^j\psi\big(\textstyle{\frac12}a,x_2\big) = \e^{i\theta a}\, \pd_1^j\psi\big(-\textstyle{\frac12}a,x_2\big), \quad j=0,1\,;
 \end{equation}
the associated quadratic form
 \begin{equation} \label{form}
q_{\alpha,\Gamma_0}[\psi,\theta] = \|\nabla\psi\|^2_{L^2(S_a)} -\alpha\int_{-a/2}^{a/2} \psi\big(\Gamma(s(\xi))\big)\,\D\xi
 \end{equation}
is defined on functions from $H^1(S_a)$ satisfying the condition \eqref{quasiper} with $j=0$. Alternatively, one can consider the fiber operators
 \begin{equation} \label{fiber2}
\tilde{H}_{\alpha,\Gamma_0}(\theta) = (-i\pd_1+\theta)^2 -\pd_2^2 -\alpha\delta(x-\Gamma(s)),
 \end{equation}
with quasiperiodic conditions \eqref{quasiper} replaced by periodic ones, unitarily equivalent to \eqref{fiber} by means of the operators $V:\, (V\psi) = \e^{-i\theta x_1}\psi(x)$, and similarly for the associated quadratic forms. The advantage of operators \eqref{fiber2} is that their domain is independent of $\theta$ so they form a type A analytic family in the sense of \cite[Sec.~VII.2.1]{Ka76}; their spectral properties following from this fact translate in a straightforward way to those of operators \eqref{fiber}.

It is useful to introduce the orthonormal basis associated with the eigenvalues of the transverse part of operator $H_{0,\Gamma_0}(\theta)$, namely
 $$ 
\mu_m(\theta) = \Big(\frac{2\pi m}{a}+\theta\Big)^2,\quad \eta_m^\theta(x_1) = \frac{1}{\sqrt{a}}\,\e^{i(2\pi m+\theta a)x_1/a}
 $$ 
(or alternatively $\tilde\eta_m^\theta(x_1) = \frac{1}{\sqrt{a}}\,\e^{2\pi imx_1/a}$ referring to $\tilde\mu_m(\theta)= \mu_m(\theta)$ and periodic boundary conditions), allowing us to understand $H_{0,\Gamma_0}(\theta)$ as a matrix differential operator with the entries $(\eta_n^\theta, H_{0,\Gamma_0}(\theta) \eta_m^\theta) = \big( -\pd_2^2 + \mu_m(\theta)\big)\,\delta_{nm}$; for its resolvent kernel we have
 $$ 
(\eta_n^\theta, \big(H_{0,\Gamma_0}(\theta)+\kappa^2\big)^{-1} \eta_m^\theta)(x_1,x'_1) = \frac{\e^{-\kappa|x_1-x'_1|}}{2\kappa}\,\delta_{nm}.
 $$ 

Since the part of $\Gamma_0$ in $S_a$ is finite, one can check in the same way as in \cite[Thm.~4.2]{BEKS94} that the singular term in \eqref{fiber2} is relatively compact with respect to $(-i\pd_1+\theta)^2-\pd_2^2$, and consequently,
 \begin{equation} \label{specess}
\sigma_\mathrm{ess}\big(H_{\alpha,\Gamma_0}(\theta)\big) = \sigma_\mathrm{ess}\big(H_{0,\Gamma_0}(\theta)\big) = [\theta^2,\infty), \quad \theta\in\BB_a:=\big[-\textstyle{\frac{\pi}{a}},\textstyle{\frac{\pi}{a}}\big].
 \end{equation}
At the same time, we have $\inf \sigma\big(H_{\alpha,\Gamma_0}(\theta)\big)<\theta^2$. Indeed, let us choose
 $$ 
\psi_\eps(x) = \eta_0^\theta(x_1)\, g(\eps x_2),
 $$ 
where the function $g\in C_0^\infty(\R)$ equals one in the vicinity of zero. Plugging it into \eqref{form}, we get
 $$ 
q_{\alpha,\Gamma_0}[\psi_\eps] = \theta^2 + \eps\|g'\|^2 -\frac{\alpha}{a}\,\int_{-a/2}^{a/2} |g(\eps\Gamma_0(s(x_1))|^2\,\sqrt{1+\gamma'(x_1)^2}\,\D x_1.
 $$ 
In the limit $\eps\to 0$, the second term on the right-hand side tends to zero, while the third one approaches $-\alpha L(\Gamma_0\!\!\upharpoonright_{S_a}) < -\alpha a<0$, where $L(\Gamma_0\!\!\upharpoonright_{S_a})$ is the length of the segment of the curve $\Gamma_0$ in the strip \eqref{strip}, hence $q_{\alpha,\Gamma_0}[\psi_\eps] <  \theta^2$ holds for all $\eps$ small enough.

The spectrum of $H_{\alpha,\Gamma_0}(\theta)$ below $\theta^2$ is thus nonempty and in view of \eqref{specess} it consists of eigenvalues which are real-analytic as functions of the quasimomentum $\theta$. Moreover, since we have excluded the trivial case of a constant $\gamma$, none of the dispersion curves is constant. To see that, assume that a dispersion curve satisfies $\epsilon'(\theta)=0$ in $\BB_a$. By Feynman-Hellmann argument we get for the corresponding eigenfunction $(\tilde\psi(\theta), 2(-i\pd_1+\theta)\tilde\psi(\theta))=0$, however, $\tilde\psi(\theta)$ is nonzero a.e., so that the eigenfunction must satisfy
 $$ 
(-i\pd_1+\theta)\tilde\psi(x_1,x_2;\theta)=0
 $$ 
for a.a. $x_2\in\R$ which can happen only if $\tilde\psi(\cdot,x_2;\theta)$ is a multiple of $\tilde\eta_0^\theta(\cdot)$. The interaction term in the operator \eqref{fiber2}, though, leaves the subspace of functions independent of $x_1$ invariant only if $\gamma'=0$, which contradicts to the assumption we have made.

Finally, the spectrum of $H_{\alpha,\Gamma_0}$ is by \eqref{floquet} union of the fiber operator \eqref{fiber} spectra as $\theta$ runs through $\BB_a$ which concludes the proof.
\end{proof}

\begin{remarks} \label{rem:ac}
(a) The lowest eigenvalue $\epsilon_0$ of $H_{\alpha,\Gamma_0}(0)$ is simple and can be chosen positive; its periodic extension is the generalized eigenfunction of $H_{\alpha,\Gamma_0}$ associated with the spectral threshold. \\[.1em]
(b) In fact, the whole spectrum is expected to be absolutely continuous but we will not need this property in the present paper. We mention it here only because the proof existing in the literature \cite{BDE03} requires substantially stronger hypotheses: the $C^4$ smoothness of $\Gamma_0$ and $\alpha$ large enough.
\end{remarks}

The question we are going to address concerns spectral effects of local perturbations of $H_{\alpha,\Gamma_0}$ coming from rescaling of the interaction support. Specifically, we ask what happens if $\Gamma_0$ is replaced by the curve
 \begin{equation} \label{rescaled}
\Gamma_\tau(s)=(x_1+\tau(x_1),\gamma(x_1))
 \end{equation}
assuming that
 \begin{enumerate}[(i)]
 \setcounter{enumi}{1}
 \setlength{\itemsep}{1.5pt}
 \item $\tau:\R\to\R$ is a Lipschitz function, piecewise $C^1$ with $\tau'$ compactly supported; we again exclude the trivial case $\tau=0$. \label{aii}
  \end{enumerate}
The variable $s$ here is no longer the arc length of the perturbed curve $\Gamma_\tau$, the latter being given by $s_\tau(x_1)= \int^{x_1}_0 \sqrt{(1+\tau'(x_1))^2 + \gamma'(\xi)^2}\,\D\xi$, however, this parametrization difference plays no role as we are going to compare points on the curves through their Euclidean distances in~$\R^2$.

\section{The singular version of Birman-Schwinger principle}
\label{s:BS}

The core element of spectral analysis of a self-adjoint operator, in our case $H_{\alpha,\Gamma_\tau}$, is the investigation of its resolvent. An efficient tool to identify the discrete spectrum of Schr\"odinger operators is the Birman-Schwinger principle; we refer to \cite{BEG22} for an extensive bibliography on this subject. If the involved potentials are singular, of measure type, a modification is needed.

Let us recall that the result comes from comparison of the resolvent, $R^k_{\alpha,\Gamma}=\big(H_{\alpha,\Gamma}-k^2)^{-1}$ for $k^2\in \rho\big(H_{\alpha,\Gamma}\big)$, with the free one, $R_0^k$, which in the two-dimensional case is an integral operator with the kernel
 $$ 
G_k(x-y) = \frac{i}{4}H_0^{(1)}(k|x-y|).
 $$ 
Given positive Radon measures $\mu,\nu$ with $\mu(x)=\nu(x)=0$ for any $x\in\R^2$, we define by $R^k_{\nu,\mu}$ the integral operator $L^2(\R^2,\mu)=: L^2(\mu)\to L^2(\nu)$ acting as $R^k_{\nu,\mu}\psi= G_k*\psi\mu$ for all $\psi\in D(R^k_{\nu,\mu})\subset L^2(\mu)$. Specifically, we deal with combinations of the measure $m_\Gamma:\, m_\Gamma(M) = \int_\R \chi_{M\cap\Gamma}(x_1)\,\sqrt{1+\gamma'(x_1)^2}\,\D x_1$ for any Borel $M\subset\R^2$, and the Lebesgue measure $\D x$. The generalized Birman-Schwinger principle is then expressed as follows \cite[Prop.~2.3]{Ex08}:
\begin{proposition} \label{prop:BS}
(i) Fix a complex $k$ with $\mathrm{Im}\,k>0$. If $\,I-\alpha R^k_{m_\Gamma,m_\Gamma}$ is invertible and the operator
 $$ 
R^k:= R^k + \alpha R^k_{\D x,m_\Gamma} \big(I-\alpha R^k_{m_\Gamma,m_\Gamma}\big)^{-1} R^k_{m_\Gamma, \D x}
 $$ 
on $L^2(\R^2)$ is everywhere defined, then $k^2\in \rho\big(H_{\alpha,\Gamma}\big)$ and $R^k=R^k_{\alpha,\Gamma}$. \\
(ii) $\dim\ker\big(H_{\alpha,\Gamma}-k^2) = \dim\ker\big(I-\alpha R^k_{m_\Gamma,m_\Gamma}\big)$ for any such $k$. \\
(iii) An eigenfunction $\psi\in L^2(\R^2)$ of $H_{\alpha,\Gamma}$ with eigenvalue $k^2$ is associated with an appropriate eigenfunction $\phi\in L^2(m_\Gamma)\backsimeq L^2(\R)$ of $R^k_{m_\Gamma,m_\Gamma}$ with eigenvalue one, being related by $\psi(x)= \int_\R R^k_{\D x,m_\Gamma}(x,s)\,\phi(s)\,\D s$, and on the other hand, $\phi(s)=\psi(\Gamma(s))$.
\end{proposition}

Instead of the Laplacian resolvent sandwiched between square roots of the potential we have in the singular case an integral operator, the kernel of which is the ($\alpha$ multiple of the) trace of the free resolvent kernel at the points of the interaction support. To be concrete, we are interested in the negative spectrum of $H_{\alpha,\Gamma}$ corresponding to $k=i\kappa$ with $\kappa>0$, thus we seek the values of $\kappa$ for which operator $\RR^\kappa_{\alpha,\Gamma}$ on $L^2(\R)$ with the kernel
 \begin{equation} \label{BSkern}
\RR^\kappa_{\alpha,\Gamma}(s,s') = \frac{\alpha}{2\pi} K_0\big(\kappa|\Gamma(s)-\Gamma(s')|\big),
 \end{equation}
where $K_0$ is the Macdonald function, has eigenvalue one.

\begin{proposition} \label{prop:BS2}
The operator $\RR^\kappa_{\alpha,\Gamma}$ is positive and decreasing with respect to the spectral parameter, $\RR^\kappa_{\alpha,\Gamma} \ge \RR^{\kappa'}_{\alpha,\Gamma}$ for $\kappa'>\kappa$, and $\lim_{\kappa\to\infty}\|\RR^\kappa_{\alpha,\Gamma}\|=0$. Moreover, $\sup\sigma_{ess}(\RR^{\kappa_0}_{\alpha,\Gamma})=1$ holds for $-\kappa_0^2=\inf\sup\sigma_{ess}(H_{\alpha,\Gamma})$.
\end{proposition}
\begin{proof}
The positivity follows from positivity of $(-\Delta+\kappa^2)^{-1}$, the monotonicity and vanishing as $\kappa\to\infty$ follow from the properties of $K_0(\cdot)$. The last claim is a consequence of the monotonicity in combination with part (ii) of the previous proposition\footnote{Note that the BS principle can be applied to the essential spectrum directly using the spectral shift function \cite{Pu11}.}.
\end{proof}

The Birman-Schwinger principle makes it possible to investigate spectral properties of $H_{\alpha,\Gamma}$ through those of $\RR^\kappa_{\alpha,\Gamma}$. To this aim, it is in our case useful to employ the unitary map
 $$ 
U:\,L^2(\R)\to\ L^2(\R), \quad (U\psi)(s) = \psi(s(x_1))\, \sqrt[4]{1+\gamma(x_1)^2},
 $$ 
and examine the operator $\tilde\RR^\kappa_{\alpha,\Gamma} = U\RR^\kappa_{\alpha,\Gamma}U^{-1}$ with the kernel
 \begin{equation} \label{Ukern}
\tilde\RR^\kappa_{\alpha,\Gamma}(x_1,x_1') = \RR^\kappa_{\alpha,\Gamma}(s(x_1),s(x_1')).
 \end{equation}

\section{The perturbed essential spectrum}
\label{s:pert}

Our aim to find what is the effect of the perturbation \eqref{rescaled} on the spectrum, in particular, below $\epsilon_0=\inf\sigma\big(H_{\alpha,\Gamma_0}\big)$. Since the unperturbed spectrum is purely essential by Proposition~\ref{prop:ess}, we have first to check that this spectral component cannot spread below $\epsilon_0$ under the hypotheses made so far; recall that assumption \eqref{ai} about the unperturbed curve was used in the indicated proposition and \eqref{aii} concerns the perturbation \eqref{rescaled}.
\begin{theorem} \label{thm:tresh}
Assuming \eqref{ai} and \eqref{aii}, we have $\inf\sigma_\mathrm{ess}\big(H_{\alpha,\Gamma_\tau}\big)=\epsilon_0$.
\end{theorem}
\begin{proof}
It is easy to see that the essential spectrum threshold cannot move up, $\inf\sigma\big(H_{\alpha,\Gamma_\tau}\big)\le\epsilon_0$. To this aim, it is sufficient to find an appropriate Weyl sequence, for instance, the generalized eigenfunction at the threshold with a suitable mollifier,
 \begin{equation} \label{Weyl}
f_m(x) = \frac{1}{\sqrt{m}}\,\psi_0(x)\,g\big(\textstyle{\frac{x_1-b_m}{m}}\big),
 \end{equation}
where $g$ is a fixed function from $C_0^\infty(\R)$. Choosing the shifts $b_m$ in such a way that the supports of $g$ and $\tau'$ are disjoint, and therefore $\Gamma_\tau=\Gamma_0$ holds on $\mathrm{supp}\,f_m$ (modulo a possible longitudinal shift), we can express
 \begin{align} \label{Weyl2}
& (H_{\alpha,\Gamma_\tau} f_m -\epsilon_0f_m)(x) \\ & \quad = -2m^{-3/2}\,(\pd_1\psi_0)(x)\,g'\big(\textstyle{\frac{x_1-b_m}{m}}\big) -m^{-5/2}\psi_0(x) g''\big(\textstyle{\frac{x_1-b_m}{m}}\big). \nonumber
 \end{align}
Next we observe that by Remarks~\ref{rem:ac}(a) $\psi_0$ is a periodic extension of a Lipschitz function from $L^2(S_a)$ which is smooth away from $\Gamma_0$ and that its derivative normal to the curve has a jump equal to $\alpha$; it follows that both $\psi_0$ and $\pd_1\psi_0$ have a uniform bound, and consequently, the norms of the two terms at the right-hand side of \eqref{Weyl2} are $\OO(m^{-1})$ and $\OO(m^{-2})$, respectively, as $m\to\infty$. Moreover, one can choose the sequence $\{b_m\}$ is such a way that the functions $f_m$ have disjoint supports in which case $f_m\to 0$ weakly as $m\to\infty$, and from that $\epsilon_0 \in \sigma\big(H_{\alpha,\Gamma_\tau}\big)$ follows.

The opposite inequality can be obtained by Neumann bracketing \cite[Sec.~XIII.15]{RS78}; it is enough to find an operator obtained from $H_{\alpha,\Gamma_\tau}$ by imposing additional Neumann condition such that the threshold of its essential spectrum is $\epsilon_0$. To begin with, we divide the plane into union (modulo the boundaries of codimension one) of five disjoint regions, namely $\Omega_1^{(\pm)}:= \{x:\,\pm x_1>(n+\frac12)a, \,x_2\in\R\}$, $\Omega_2^{(\pm)}:= \{x:\, |x_1|<(n+\frac12)a, \pm x_2>\|\gamma\|_\infty\}$, and $\Omega_0:= \{x:\, |x_1|<(n+\frac12)a, |x_2|<\|\gamma\|_\infty\}$; the number $n$ is chosen in such a way that $\mathrm{supp}\,\tau \subset \big(-(n+\frac12)a,(n+\frac12)a\big)$ in which case $\Gamma_\tau$ coincides with $\Gamma_0$ in $\Omega_1^{(\pm)}$ up to a possible shift. Neumann conditions at the common boundaries of the domains allow us to estimate or operator from below,
 \begin{equation} \label{Nbrack}
H_{\alpha,\Gamma_\tau} \ge H_{\alpha,\Gamma_\tau} = H_{\alpha,\Gamma_\tau}^{(1,+)} \oplus H_{\alpha,\Gamma_\tau}^{(1,-)} \oplus H_{\alpha,\Gamma_\tau}^{(1,-)} \oplus H_{\alpha,\Gamma_\tau}^{(2,+)} \oplus H_{\alpha,\Gamma_\tau}^{(2,-)} \oplus H_{\alpha,\Gamma_\tau}^{(0)}.
 \end{equation}
The last three terms of this decomposition are irrelevant as far as the essential spectrum threshold is concerned, since $H_{\alpha,\Gamma_\tau}^{(2,\pm)}\ge 0$ and we know that $\epsilon_0<0$, and $\Omega_0$ is bounded so that the spectrum of $H_{\alpha,\Gamma_\tau}^{(0)}$ is purely discrete.

The estimate \eqref{Nbrack} yields the result if the function $\gamma$ is mirror-symmetric; without loss of generality we may suppose that the symmetry point is at the origin, $\gamma(x_1)= \gamma(-x_1)$. Indeed, in that case we have $\pd_1\psi_0(0,x_2)=0$ for any $x_2\in\R$ and $\inf\sigma\big(H_{\alpha,\Gamma_0}\big) =\inf\sigma\big(H_{\alpha,\Gamma_0}^\mathrm{N}\big)$, where the last operator has an additional Neumann condition imposed at the line $x_1=0$, however, the said operator is unitarily equivalent to $H_{\alpha,\Gamma_\tau}^{(1,+)} \oplus H_{\alpha,\Gamma_\tau}^{(1,-)}$.

In the absence of the symmetry this argument no longer works and have to take a closer look at the spectral threshold of the operators $H_{\alpha,\Gamma_\tau}^{(1,\pm)}$. For definiteness, we consider the plus sign, the argument for the other is analogous. In the same spirit as above,
 $$ 
H_{\alpha,\Gamma_\tau}^{(1,+)} \ge \sum_{m\in\N_0}\!\!\raisebox{1.5ex}{{\scriptsize $\oplus$}} H_{\alpha,\Gamma_\tau}^{(1,m)},
 $$ 
where the orthogonal sum terms refer to restrictions of $H_{\alpha,\Gamma_\tau}^{(1,m)}$ to the regions $\Omega_{1,m}:= \{x:\,(n+m(2j+1)+\frac12)a < x_1 < (n+(m+1)(2j+1)+\frac12)a,\,x_2\in\R\}$ with Neuman boundary conditions. They are all unitarily equivalent by shifts on multiples of $(2j+1)a$ in the $x_1$ direction, hence we can consider one of them only, and this is in turn unitarily equivalent to the Neumann restriction $H_{\alpha,\Gamma_0}^{[j]}$ of $H_{\alpha,\Gamma_0}$ to $\tilde\Omega_j:= \{x:\, |x_1|<(j+\frac12)a,\, x_2\in\R\}$ for any fixed $j\in\N_0$. Using Neumann bracketing inductively, on checks easily that $\epsilon[j]:= \inf\sigma\big(H_{\alpha,\Gamma_0}^{[j]}\big)$ is nondecreasing with respect to $j$ and bounded from above by $\epsilon_0$, hence the limit $\epsilon_\infty = \lim_{j\to\infty} \epsilon[j]$ exists and $\epsilon_\infty < \epsilon_0$. Consider next the family of quadratic forms with the common domain $H^1(\R^2)$ defined by
 $$ 
q^{(j)}_{\alpha,\Gamma_0}[\psi] = \int_{\tilde\Omega_j} |\nabla\psi(x)|^2\,\D x - \alpha \int_\R |\psi(\Gamma_0(s))|^2\,\D s,
 $$ 
clearly associated with the orthogonal sum of operators, $H_{\alpha,\Gamma_0}^{[j]} \oplus (-\alpha I)$ on $L^2(\tilde\Omega_j) \oplus L^2(\R^2\setminus\tilde\Omega_j)$. In view of the absolute continuity of the Lebesgue integral we have $\lim_{j\to\infty} q^{(j)}_{\alpha,\Gamma_0}[\psi] = q_{\alpha,\Gamma_0}[\psi]$ for any $\psi\in H^1(\R^2)$, and this in view of the form family monotonicity implies that $H_{\alpha,\Gamma_0}^{[j]} \oplus (-\alpha I) \to H_{\alpha,\Gamma_0}$ in the strong resolvent sense, cf.~\cite[Sec.~VIII.3.4]{Ka76} or \cite[Thm.~10.12]{Si18}. In this way, we arrive at $\epsilon_\infty = \epsilon_0$ which concludes the proof.
\end{proof}

Theorem~\ref{thm:tresh} is what we will need in the next section. One expects that under the assumptions of the said theorem the whole essential spectrum will be preserved. For a subset of local perturbations which one may label as `critical', specified by the additional requirement
 \begin{enumerate}[(i)]
 \setcounter{enumi}{2}
 \setlength{\itemsep}{1.5pt}
 \item $\int_\R \tau'(x_1)\,\D x_1 = 0$, \label{aiii}
  \end{enumerate}
there are simple ways how to prove the invariance:
\begin{proposition} \label{prop:essinv}
Under \eqref{ai}--\eqref{aiii}, we have $\sigma_\mathrm{ess}\big(H_{\alpha,\Gamma_\tau}\big)=\sigma_\mathrm{ess}\big(H_{\alpha,\Gamma_0}\big)$.
\end{proposition}
\begin{proof}
It is straightforward to see that $\sigma_\mathrm{ess}\big(H_{\alpha,\Gamma_\tau}\big) \supset\sigma_\mathrm{ess}\big(H_{\alpha,\Gamma_0}\big)$ using suitable Weyl sequences. For the nonnegative $k^2\in \sigma_\mathrm{ess}\big(H_{\alpha,\Gamma_0}\big)$ its elements can be plane waves of momentum $k$ with compactly supported mollifiers, for the energies $0>\epsilon_0+\kappa^2\in \sigma_\mathrm{ess}\big(H_{\alpha,\Gamma_0}\big)$ we repeat the argument from the previous proof replacing $\psi_0$ in \eqref{Weyl} by the periodic extension of the function $\psi_\kappa$ satisfying $H_{\alpha,\Gamma_0}(\theta)\psi_\kappa = (\epsilon_0+\kappa^2) \psi_\kappa$.

To prove the opposite inclusion, we can use Weyl's criterion again; as before we may consider the negative part of $\sigma_\mathrm{ess}\big(H_{\alpha,\Gamma_\tau}\big)$ only. To any element $\epsilon_0+\kappa^2$ of it, there is a sequence $\{\psi_m\}\subset H^2(\R^2)$ such that $\psi_m\to 0$ weakly and $\|(H_{\alpha,\Gamma_\tau}-\epsilon_0-\kappa^2)\psi_m\|\to 0$ as $m\to\infty$. The point is that under \eqref{aiii}, $\{\psi_m\}$ is a Weyl sequence for $H_{\alpha,\Gamma_0}$ as well. Indeed, we have
$$ 
\|(H_{\alpha,\Gamma_0}-\epsilon_0-\kappa^2)\psi_m\| \le \|(H_{\alpha,\Gamma_\tau}-\epsilon_0-\kappa^2)\psi_n\| + \|(H_{\alpha,\Gamma_0}-H_{\alpha,\Gamma_\tau})\psi_n\|.
$$ 
The first term on the right-hand side tend to zero as $m\to\infty$ by assumption, the second can be written explicitly as $\big(\int_\R |\psi_m(\Gamma_0(s)) - \psi_m(\Gamma_\tau(s))|^2\,\D s \big)^{1/2}$, where the integrated function is continuous and compactly supported. The weak convergence then implies $\psi_m$ converges to zero pointwise being uniformly bounded \cite[Prop.~19.3.1]{Se71}, and as a result, this term vanishes in the limit too yielding thus the sought inclusion.

There is an \emph{alternative way}\footnote{The idea belongs to the anonymous referee of the paper \cite{ES25}.} to prove the claim using the following `abstract nonsense' argument. Let $B$ be relatively bounded with respect to a self-adjoint $A_0$ with the bound less than one so that $A=A_0+B$ is also self-adjoint. Furthermore, let $\{P_m\}$ be a family of projections such that $\|BP_m^\perp\|\to 0$ as $m\to\infty$. If $P_m(A_0+i)^{-1}$ is compact for all $m\ge 1$, then $(A+i)^{-1}-(A_0+i)^{-1}$ is compact as one can see from the identity
 \begin{equation} \label{anonsense}
(A+i)^{-1}-(A_0+i)^{-1} = (A+i)^{-1}BP_m^\perp(A_0+i)^{-1} + (A+i)^{-1}BP_m(A_0+i)^{-1}.
 \end{equation}
Indeed, the operator $(A+i)^{-1}B$ is bounded because its adjoint is bounded by assumption, which implies that the second term on the right-hand side of \eqref{anonsense} is compact for all $m$. The first one, on the other hand, converges to zero in the operator norm; the second factor is compact for any $m$ and its norm limit as $m\to\infty$ is thus a compact operator again. This result can be applied to our case, where $A_0=H_{\alpha,\Gamma_0}$ and $B$ is the difference of the singular terms corresponding to $\Gamma_\tau$ and $\Gamma_0$. Choosing, for instance, projections $P_m$ with $\mathrm{Ran}\,P_m = L^2(B_m(0))$, we have $BP_m^\perp=0$ for all $m$ large enough. At the same time, we know that $D(H_{\alpha,\Gamma_0})=H^2(\R^2)$, hence the needed compactness of operator $P_m(H_{\alpha,\Gamma_0}+i)^{-1}$ is a consequence of Rellich-Kondrashov theorem \cite[Theorem~{6.3}]{AF03}.
\end{proof}

\section{Discrete spectrum due to local perturbations}
\label{s:discrete}

To find whether the perturbation can create a spectrum below the threshold of the essential one, we are going to compare Birman-Schwinger operators referring to $\Gamma_0$ and $\Gamma_\tau$. We follow the idea proposed in \cite{EI01}: by Proposition~\ref{prop:BS2} both $\sup\sigma_\mathrm{ess}\big(\RR^\kappa_{\alpha,\Gamma}\big)$ and the possible eigenvalues above it decrease continuously to zero as $\kappa\to\infty$. We have thus to check the inequality
 \begin{equation} \label{excondition}
\sup\sigma\big(\RR^{\kappa_0}_{\alpha,\Gamma_\tau}\big) > \sup\sigma_\mathrm{ess}\big(\RR^{\kappa_0}_{\alpha,\Gamma_0}\big)
 \end{equation}
for $\kappa_0:= \big(\!-\inf\sigma\big(H_{\alpha,\Gamma_0}\big)\big)^{1/2}$, since if it holds, there is at least one eigen\-value $\mu(\kappa)$ of $\RR^\kappa_{\alpha,\Gamma_\tau}$ above the essential spectrum threshold which crosses the value one \emph{to the right} of the point $\kappa_0$, and consequently, the discrete spectrum of the operator is by Proposition~\ref{prop:BS}(ii) nonempty.

As one might expect, this happens if the perturbation leads to a local contraction of the interaction support:
\begin{theorem} \label{thm:contract}
Under \eqref{ai} and \eqref{aii}, $\sigma_\mathrm{disc}\big(H_{\alpha,\Gamma_\tau}\big)\ne\emptyset$ holds if $\tau'\le 0$.
\end{theorem}
\begin{proof}
To verify inequality \eqref{excondition}, it is sufficient to find a suitable trial function. The natural starting point is the generalized eigenfunction $\psi_0$ corresponding to $\inf\sigma\big(H_{\alpha,\Gamma_0}\big)$; by Remark~\ref{rem:ac}(a) it is simple and one choose it positive and by Proposition~\ref{prop:BS}(iii) the corresponding generalized eigenfunction of $\RR^\kappa_{\alpha,\Gamma_0}$ is $\phi_0(s)=\psi_0(\Gamma_0(s))$, or alternatively $\tilde\phi_0=U\phi_0$ referring to the operator with the kernel \eqref{Ukern}.

Such a function does not belong to $L^2$, of course, so we have to amend it with a suitable mollifier and ensure that the effect of the latter can be made arbitrarily small. Specifically, we want to prove that one can choose a sequence $\{g_n\}\subset L^1(\R)$ with the property
 \begin{equation} \label{mollif1}
\big(g_n\tilde\phi_0, \tilde\RR^{\kappa_0}_{\alpha,\Gamma_0}g_n\tilde\phi_0\big) - \|g_n\tilde\phi_0\|^2 = \OO(n^{-2})
 \end{equation}
holds as $n\to\infty$. We denote the left-hand side of \eqref{mollif1} as $M_n$; using the fact that by Proposition~\ref{prop:BS2} we have $\tilde\RR^{\kappa_0}_{\alpha,\Gamma_0}\tilde\phi_0 =\tilde\phi_0$, we can rewrite the second term as $\big(g_n\tilde\phi_0, g_n\tilde\RR^{\kappa_0}_{\alpha,\Gamma_0}\tilde\phi_0\big)$ so that
 $$ 
M_n= \big(g_n\tilde\phi_0, \big[\tilde\RR^{\kappa_0}_{\alpha,\Gamma_0},g_n\big]\tilde\phi_0\big),
 $$ 
or explicitly
 \begin{align}
M_n= & \frac{\alpha}{2\pi} \int_{\R^2} g_n(x_1) \tilde\phi_0(x_1)\, K_0\big(\kappa_0|\Gamma_0(s(x_1))-\Gamma_0(s(x'_1))|\big) \nonumber \\ & \times (g_n(x'_1)-g_n(x_1))\, \tilde\phi_0(x'_1)\,\D x_1\D x'_1. \label{mollif3}
 \end{align}
To estimate the right-hand side of \eqref{mollif3} we note that the function $\tilde\phi_0$ is bounded, and moreover, the inequality
 $$ 
|\Gamma_0(s(x_1))-\Gamma_0(s(x'_1))| = \sqrt{(x_1-x'_1)^2 + (\gamma(x_1)-\gamma(x'_1)^2} \ge |x_1-x'_1|
 $$ 
implies
 $$ 
K_0\big(\kappa_0|\Gamma_0(s(x_1))-\Gamma_0(s(x'_1))|\big) \le K_0(\kappa_0|x_1-x'_1|)
 $$ 
in view of the monotonicity of $K_0(\cdot)$. Choosing then
 \begin{equation} \label{mollif7}
g_n(x_1) = \frac{n^2}{n^2+x_1^2}
 \end{equation}
we get the estimate
 $$ 
M_n \le \frac{\alpha}{2\pi}\,n^2\, \|\tilde\phi_0\|^2_\infty \int_{\R^2} \frac{|x_1+x'_1||x_1-x'_1|}{(n^2+x_1^2)^2 (n^2+{x'_1}^2)}\, K_0(\kappa_0|x_1-x'_1|)\,\D x_1\D x'_1
 $$ 
Passing next to the variables $u=x_1-x'_1$ and $u=x_1-x'_1$, we obtain
 \begin{align*}
& M_n \le \frac{\alpha}{2\pi}\,n^2\, \|\tilde\phi_0\|^2_\infty \int_{\R^2} \frac{|u||v|}{\big(n^2+\frac14(u+v)^2\big)^2 \big(n^2+\frac14(u-v)^2\big)}\, K_0(\kappa_0|u|)\,\D u\D v \\
& = \frac{2\alpha}{\pi}\,n^2\, \|\tilde\phi_0\|^2_\infty \int_0^\infty \!\! \int_0^\infty \frac{u^2 v^2}{\big(n^2+\frac14(u+v)^2\big)^2 \big(n^2+\frac14(u-v)^2\big)^2}\, K_0(\kappa_0|u|)\,\D u\D v,
 \end{align*}
and since the denominator is bounded from below by $\frac{1}{16}\,(1+v^2)^2\,n^4$ for any $n\ge 1$ and the function $u\mapsto u^2\,K_0(\kappa_0|u|)$ is bounded and exponentially decaying as $u\to\infty$, we arrive finally at the inequality
 $$ 
M_n \le \frac{32\alpha}{\pi}\,n^{-2}\, \|\tilde\phi_0\|^2_\infty \int_0^\infty \frac{v^2}{(1+v^2)^2}\,\D v\, \int_0^\infty u^2\,K_0(\kappa_0|u|)\,\D u
 $$ 
which yields the sought relation \eqref{mollif1}.

To prove that \eqref{excondition} holds under the assumptions of the theorem, we have to check that
 $$ 
\big(g_n\tilde\phi_0, \tilde\RR^{\kappa_0}_{\alpha,\Gamma_\tau}g_n\tilde\phi_0\big) - \|g_n\tilde\phi_0\|^2 >0
 $$ 
holds for some $n$. Adding and subtracting $\big(g_n\tilde\phi_0, \tilde\RR^{\kappa_0}_{\alpha,\Gamma_0}g_n\tilde\phi_0\big)$ to the left-hand side we see that this will happen provided
 \begin{equation} \label{trial_lim}
\lim_{n\to\infty} \big[\big(g_n\tilde\phi_0, \tilde\RR^{\kappa_0}_{\alpha,\Gamma_\tau}g_n\tilde\phi_0\big) - \big(g_n\tilde\phi_0, \tilde\RR^{\kappa_0}_{\alpha,\Gamma_0}g_n\tilde\phi_0\big)\big] >0
 \end{equation}
because then the difference will for all large enough $n$ exceed the effect of the mollifier. Using \eqref{BSkern} and \eqref{Ukern}, we can write both terms in \eqref{trial_lim} explicitly; using then \eqref{mollif7} and the dominated convergence theorem, the condition becomes
 \begin{align}
\int_{\R^2} \tilde\phi_0(x_1)\, & \big[K_0\big(\kappa_0|\Gamma_\tau(s(x_1))-\Gamma_\tau(s(x'_1))|\big) \nonumber \\ & - K_0\big(\kappa_0|\Gamma_0(s(x_1))-\Gamma_0(s(x'_1))|\big)\big]\, \tilde\phi_0(x'_1)\,\D x_1\D x'_1 > 0. \label{trial_lim2}
 \end{align}
It is now easy to see that \eqref{trial_lim2} holds not only for $\kappa_0$ but in fact for all $\kappa>0$. Indeed, the argument of the kernel contains
 $$ 
|\Gamma_\tau(s(x_1))-\Gamma_\tau(s(x'_1))| = \sqrt{(x_1 +\tau(x_1)-x'_1 -\tau(x'_1))^2 + (\gamma(x_1)-\gamma(x'_1)^2}
 $$ 
so that $|\Gamma_\tau(s(x_1))-\Gamma_\tau(s(x'_1))| < |\Gamma_0(s(x_1))-\Gamma_0(s(x'_1))|$ holds on an open subset of $\R^2$ outside which the two expressions are equal; recall that we have excluded the trivial case $\tau=0$. The indicated conclusion then follows from the monotonicity of $K_0(\cdot)$ and positivity of the function $\tilde\phi_0$.
\end{proof}

If $\tau'$ is sign-changing, the situation is more involved, nevertheless, one can get an existence result in the situation where both the undulation of $\Gamma_0$ and the perturbation are gentle and regular enough. To be specific, we consider for given $\tau$ and $\gamma$ the following family of curves,
 $$ 
\Gamma^\varepsilon_{\varepsilon\tau}(s)=(x_1+ \varepsilon\tau(x_1),\varepsilon\gamma(x_1)),\;\; \varepsilon>0.
 $$ 
Under such an assumption, we can even consider the `critical' case when the perturbation has zero mean:
\begin{theorem} \label{thm:gentlecrit}
In addition to assumptions \eqref{ai}--\eqref{aiii} made above, suppose that $\tau,\gamma\in C^2$ and $\tau'(0)=0$; then there is at least one eigenvalue below the essential spectrum bottom, so that $\sigma_\mathrm{disc}\big(H_{\alpha,\Gamma^\varepsilon_{\varepsilon\tau}}\big)\ne\emptyset$, for all~$\varepsilon$ small enough.
\end{theorem}
\begin{proof}
The argument of the previous proof applies naturally again, hence one has to check whether
 \begin{align}
\int_{\R^2} \tilde\phi_0(x_1)\, & \big[K_0\big(\kappa_0|\Gamma^\varepsilon_{\varepsilon\tau}(s(x_1))-\Gamma^\varepsilon_{\varepsilon\tau}(s(x'_1))|\big) \nonumber \\ & - K_0\big(\kappa_0|\Gamma_0(s(x_1))-\Gamma_0(s(x'_1))|\big)\big]\, \tilde\phi_0(x'_1)\,\D x_1\D x'_1 > 0. \nonumber 
 \end{align}
holds. The difference of the two distances as a function of $x_1$ now changes sign, but one could still verify positivity of the integral as long as the function involved is convex using the integral version of Jensen's inequality. The trouble is that it may not be the case because the function $R_\kappa:\: R_\kappa(x_1)= K_0(\kappa z(x_1))$ is a map composed of the Macdonald function which is convex and $z(x_1)= \sqrt{(x_1 +\tau(x_1))^2 + \gamma(x_1)^2}$ which typically is not. To check under which condition $R_\kappa$ is (strictly) convex we proceed as in Lemma~4.5 of \cite{ES25}: evaluating $\frac{\D^2}{\D x_1^2} R_\kappa(x_1)$ and using recurrent relations between Macdonald functions \cite[9.6.26]{AS}, we arrive at the condition
 \begin{equation} \label{convexcond}
\frac{z''}{(z')^2} < \frac{1}{z} + \kappa \frac{K_0(\kappa z)}{K_1(\kappa z)}\,;
 \end{equation}
the second term on the right-hand side is positive, hence the strict convexity is guaranteed provided $zz''<(z')^2$. This may not be the case in general, however, using $z_\varepsilon(x_1)= \sqrt{(x_1 +\tau_\varepsilon(x_1))^2 + \gamma_\varepsilon(x_1)^2}$ and keeping the terms up to the first order in $\varepsilon$, we arrive at the inequality
 $$ 
x_1^2 +2\varepsilon\big(\tau(x_1)+ \tau'(x_1)\big) + \OO(\varepsilon^2) > 0
 $$ 
which is certainly satisfied independently of $x_1$ for all $\varepsilon$ small enough; recall that the support of $\tau$ is compact by assumption.
\end{proof}

\begin{corollary} \label{cor:gentlecontr}
The claim of Theorem~\ref{thm:gentlecrit} remains valid if assumption \eqref{aiii} is replaced by \mbox{$\int_\R \tau'(x_1)\,\D x_1 < 0$.}
\end{corollary}
\begin{proof}
To such a $\tau$ one can find an $\upsilon\ge\tau$ of zero mean for which $\sigma_\mathrm{disc}\big(H_{\alpha,\Gamma^\varepsilon_{\varepsilon\upsilon}}\big)$ is nonempty provided $\varepsilon$ is small enough. We use that same trial function $g_n\tilde\phi_0$ with $n$ sufficiently large. To the difference in question,
 \begin{equation} \label{trial_diff2}
\big(g_n\tilde\phi_0, \tilde\RR^{\kappa_0}_{\alpha,\Gamma^\varepsilon_{\varepsilon\tau}}g_n\tilde\phi_0\big) - \big(g_n\tilde\phi_0, \tilde\RR^{\kappa_0}_{\alpha,\Gamma_0}g_n\tilde\phi_0\big)
 \end{equation}
we add and subtract $\big(g_n\tilde\phi_0, \tilde\RR^{\kappa_0}_{\alpha,\Gamma^\varepsilon_{\varepsilon\upsilon}}g_n\tilde\phi_0\big)$. One of the two differences obtained is positive for small $\varepsilon$ by the previous proof, while the remaining one
 $$ 
\big(g_n\tilde\phi_0, \tilde\RR^{\kappa_0}_{\alpha,\Gamma^\varepsilon_{\varepsilon\tau}}g_n\tilde\phi_0\big) - \big(g_n\tilde\phi_0,  \tilde\RR^{\kappa_0}_{\alpha,\Gamma^\varepsilon_{\varepsilon\upsilon}}g_n\tilde\phi_0\big),
 $$ 
can be treated as in the proof of Theorem~\ref{thm:contract}; in view of the inequality $\upsilon\ge\tau$ it is nonnegative, irrespective of $\varepsilon$, hence the difference \eqref{trial_diff2} is positive for $\varepsilon$ small enough.
\end{proof}

\begin{remark} \label{rem:strong}
These results are certainly far from optimal. One reason is the rough estimate we used in \eqref{convexcond} where we have neglected the positive second term on the right-hand side. The latter is not larger than $\kappa$ as mistakenly indicated in \cite[Rem.~4.6]{ES25} but approaches this value as $\kappa z\to\infty$, hence the claim made there -- that the geometric restrictions become weaker if the interaction is strong -- remains valid. The same is expected here, however, instead of working out this argument, we will present an example showing that a `critical' deformation can give rise to a nonempty discrete spectrum even if it is not weak, and neither is the undulation of $\Gamma$.
\end{remark}

\section{The perturbation may not be gentle}
\label{s:example}

Before coming to the message of this section announced in the title we will present an auxiliary result which could be of independent interest. Let us first describe the setting. We consider a positive, compactly supported function $V\in L^2(\R)$ with $\mathrm{supp}\,V\subset\Sigma_b:=[-\frac12b,\frac12b]$ for some $b>0$. Furthermore, let $X_0=\{na\}_{n\in\Z}$ be a periodic array with $a>b$ and $X_\delta=\{na+\delta_n\}_{n\in\Z}$ be a sequence obtained by shifting a \emph{finite} number of the array points; we denote $x_n=na+\delta_n$. We do not impose restrictions on the \emph{individual} $\delta_n$'s but we require that $\delta_{n+1}-\delta_n > -b$, so that the perturbation does not change the order of the points and, in addition, we have $x_{n+1}-x_n> a-b$. The object of our interest is the one-dimensional Schr\"odinger operator
 $$ 
H_{V,X_\delta} = -\frac{\D^2}{\D x^2} - V_\delta(x), \quad V_\delta = \sum_{n\in\Z} V(x-x_n),
 $$ 
describing motion in the array of potential wells which do not overlap by assumption; it is a self-adjoint operator bounded from below with the domain $H^2(\R)$ because the potential belongs to Kato class $K_1$ \cite[Thm.~A.2.7]{Si82}. It is well known that the spectrum of $H_{V,X_0}$ is purely essential and the compactly supported perturbation $V_\delta-V_0$ does not change its essential component, and it gives rise to \emph{at most} finite number of eigenvalues in each spectral gap of $\sigma\big(H_{V,X_0}\big)$ \cite{RB73, GS93}. Our question is whether such an eigenvalue does indeed exist in the lowest gap, i.e. below the spectral threshold $\epsilon_0$ of $H_{V,X_0}$; the answer is affirmative:

\begin{theorem} \label{thm:1Ddiscrete}
In the described situation, $\inf\sigma\big(H_{V,X_\delta}\big) < \epsilon_0$ unless $\delta_n=0$ for all $n\in\Z$, in other words, any nontrivial perturbation of the described type creates at least one eigenvalue below $\inf\sigma_\mathrm{ess}\big(H_{V,X_\delta}\big)$.
\end{theorem}
\begin{proof}
The argument follows the same line as in the multidimensional case \cite[Thm.~4.2]{ES25} but to make the paper self-contained, we present it skipping just a few details. The question is by Birman-Schwinger principle translated into analysis of the operator $K_{V,X_\delta}(-\kappa^2):= V_\delta^{1/2} \big(-\frac{\D^2}{\D x^2} +\kappa^2\big)^{-1} V_\delta^{1/2}$; in analogy with Proposition~\ref{prop:BS2} one can check that it a bounded positive operator, monotonous with respect to $\kappa$ and satisfying $\lim_{\kappa\to\infty} \|K_{V,X_\delta}(-\kappa^2)\|=0$. Since the potential well supports do not overlap, we may regard $K_{V,X_\delta}(-\kappa^2)$ as an infinite-matrix integral operator the entries of which map of $L^2(\Sigma_b+x_j)$ to $L^2(\Sigma_b+x_i)$ with the kernels
 \begin{equation} \label{1Dkern}
K_{V,X_\delta}^{(i,j)}(-\kappa^2;x,x') = \frac{1}{2\kappa}\,V_\delta(x)^{1/2}\, \e^{-\kappa|x-x'|}\, V_\delta(x')^{1/2}, \quad i,j\in\Z,
 \end{equation}
where $\Sigma_b+y:= \{\xi+y:\, \xi\in\Sigma_b\}$. As in the proof of Theorem~\ref{thm:contract}, to verify the claim one has to demonstrate that $\sup\sigma\big(K_{V,X_\delta}(-\kappa_0^2)\big)>1$ holds for $\kappa_0:= \sqrt{-\epsilon_0}$. The trial function we need is constructed from the generalized eigenfunction $\phi_0$ associated with $\sup\sigma\big(K_{V,X_0}(-\kappa_0^2)\big)$ which is nothing but $V_0^{1/2}\psi_0$ where $\psi_0$, is the (positive) generalized eigenfunction of $-\frac{\D^2}{\D x^2} - V_0(x)$. The function $\phi_0$ is $a$-periodic and, with an abuse of notation, we employ the same symbol for the function
 $$ 
\{\phi_{0,j}\} \in \sum_{j\in\Z}\!\raisebox{1.5ex}{{\scriptsize $\oplus$}}\: L^2(\Sigma_b+x_j), \quad \phi_{0,j}(\xi+x_j):=\phi_0(\xi)\;\;\text{for}\;\, \xi\in \Sigma_b
 $$ 
(locally aperiodic when used in expressions where the operator $K_{V,X_\delta}(-\kappa^2)$ is applied to it!). To make it an $L^2(\R)$ element, we again need mollifiers; a suitable family consists of step functions,
 \begin{equation} \label{1Dkern}
h_n(x) = \sum_{i\in\Z} h_{n,i} \chi_{\Sigma_b+x_i}(x), \quad h_{n,i} = \frac{n^2}{n^2+i^2}, \quad i\in\Z\,.
 \end{equation}
Since in view of \eqref{1Dkern} $|\big(\phi_{0,i},K_{V,X_0}(-\kappa_0^2)\phi_{0,j}\big)| \le c\,\e^{-\kappa a |i-j|}$ holds for some $c>0$, one can mimick the proof of Lemma~4.3 of \cite{ES25} and conclude that
$$ 
(h_n\phi_0, K_{V,X_0}(-\kappa_0^2)h_n\phi_0) - \|h_n\phi_0\|^2=\OO(n^{-2})\;\; \text{as}\;\; n\to\infty\,;
$$ 
using then a telescopic estimate as in the proof of Theorem~\ref{thm:contract} we see that to get the claim, we have check that
 $$ 
\lim_{n\to\infty} (h_n\phi_0, K_{V,X}(-\kappa^2)h_n\phi_0) - (h_n\phi_0, K_{V,X_0}(-\kappa^2)h_n\phi_0) > 0
$$ 
holds for $\kappa=\kappa_0$ (or even for any $\kappa>0$). As in the multidimensional case one can justify the use of dominated convergence theorem by which the problem reduces to checking positivity of the series
 $$ 
\sum_{i,j\in\Z}\big(\phi_0, \big[K_{V,X}^{(i,j)}(-\kappa^2) - K_{V,X_0}^{(i,j)}(-\kappa^2)\big]\phi_0\big),
$$ 
which we can using \eqref{1Dkern} rewrite explicitely as
 \begin{align}
\int_{\Sigma_b\times\Sigma_b} \phi_0(\xi) V^{1/2}(\xi) \sum_{i,j\in\Z} & \big[ R_\kappa(x_i-x_j +\xi-\xi') - R_\kappa((i-j)a+\xi-\xi') \big] \nonumber \\[.3em] & \times V^{1/2}(\xi') \phi_0(\xi') \,\D\xi\,\D\xi', \label{1Dexplicit}
\end{align}
where $R_\kappa(z)=\frac{1}{2\kappa}\,\e^{-\kappa|z|}$. The first term argument in the square bracket is $(i-j)a +\delta_i-\delta_j +\xi-\xi'$ with the sum of the shift terms over the indices $i,j$ being obviously zero, and since $R_\kappa(\cdot)$ is strictly convex, the bracket is positive unless all the $\delta_i$ are zero; this concludes the proof.
\end{proof}

\begin{remark} \label{rem:1versN}
Note that the result is stronger than its multidimensional analogue \cite[Thm.~4.2]{ES25} because the shift variables enter the argument of \eqref{1Dexplicit} directly and not through the composition with a concave function representing the Euclidean distance.
\end{remark}

After this preliminary let us turn to the message of this section:
\begin{example} \label{ex:strong}
Let us denote $\Sigma_a:=[-\frac12a,\frac12a]$ and consider a curve $\Gamma_0$ such that $|\mathrm{supp}\,\gamma\upharpoonright_{\Sigma_a}\!|<a$, in other words, its curved parts are separated by straight segments. Consider further a perturbed curve $\Gamma_\tau$ with the property that $\mathrm{supp}\,\tau \cap \mathrm{supp}\,\gamma = \emptyset$ which means that the shape of the curved parts does not change but their distances are locally modified, cf.~Fig~\ref{fig:localpert}; without loss of generality we may suppose that $\tau(\cdot)$ is a compactly supported step function assuming a finite number of nonzero values outside $\mathrm{supp}\,\gamma$.
 \begin{figure}[t]
 \centering
 \includegraphics[clip, trim=4cm 22cm 4cm 4cm, angle=0, width=\textwidth]{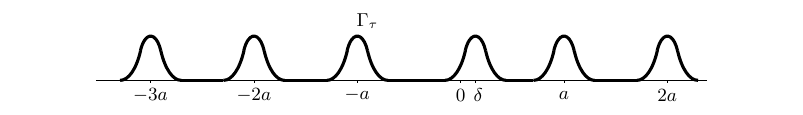}
 \caption{The interaction support $\Gamma_\tau$ in the example with one curved segment shifted}
 \label{fig:localpert}
 \end{figure}

In addition to \eqref{ai}--\eqref{aiii}, assume now that $\gamma\in C^6(\R)$ and the coupling parameter $\alpha$ in \eqref{Hform} is large. In that case, one can use the strong coupling asymptotics of the singular Schr\"odinger operators \eqref{formal} defined through the quadratic form \eqref{Hform}, cf.~\cite[Sec.~4]{Ex08}. The said asymptotics is worked out there for several classes of curves including the periodic ones; their local perturbation are not included but the technique on which the results are based, a bracketing in combination with the min-max principle, can be applied to our present problem, an eigenvalue below the bottom of the essential spectrum, as well: the conclusion is that the low-lying spectrum of $H_{\alpha,\Gamma}$ coincides, up to an error of order $\OO(\alpha^{-1}\ln\alpha)$, with spectrum of the one-dimensional Schr\"odinger operator
 \begin{equation} \label{strong_asymp}
-\frac{\D^2}{\D s^2} -\frac14\alpha^2 - \frac14 k^2(s),
\end{equation}
where $k(\cdot)$ is the signed curvature of $\Gamma$. Up to the overall spectral shift $-\frac14\alpha^2$, the operator \eqref{strong_asymp} is of the type considered in Theorem~\ref{thm:1Ddiscrete} above; note the potential modification there is compactly supported which is a requirement equivalent to assumption~\eqref{aiii}. Moreover, the curvature-induced potential is easily computed,
 $$ 
- \frac14 k^2(s) = - \frac{\gamma''(x)^2}{4(1+\gamma'(x)^2)^3},
$$ 
so that it is $C^4$ satisfying thus the regularity property needed for the bracketing method reviewed in \cite{Ex08}. This allows us to conclude that the discrete spectrum $H_{\alpha,\Gamma}$ below the threshold of $\sigma_\mathrm{ess}\big(H_{\alpha,\Gamma}\big)$ is in this case nonempty for all coupling strengths $\alpha$ large enough.
\end{example}

\section{Concluding remarks}
\label{s:concl}

As is usually the case, a resolved problem opens further questions, and for sure not only technical ones like replacing the compactly supported perturbations by those with an appropriate decay. For instance,  $\sigma\big(H_{\alpha,\Gamma_0}\big)$ may have gaps -- in fact, this happens for all sufficiently large $\alpha$, cf.~\cite{Yo98} -- and one is interested in the perturbation-induced discrete spectrum in them.

Other natural question concern sufficient conditions for the absence of the discrete spectrum -- at the bottom this may be a perturbation opposite to that of Theorem~\ref{thm:contract}, a local expansion of the curve -- or the three-dimensional analogue of the present problem in which the interaction term is no longer of additive type \cite[Sec.~2.4]{Ex08}. For the moment, however, we stop and leave these and other questions to a future research.

\subsection*{Acknowledgements}
Thanks are due to the referee for helpful remarks.

\end{document}